\documentclass{amsart}
\usepackage[latin1]{inputenc}
\usepackage[active]{srcltx}
\usepackage{hyperref}
\usepackage{amsfonts,enumerate}

\usepackage[mathscr]{euscript}
\usepackage{epsfig}
\usepackage{latexsym}

\usepackage[all]{xy}

\usepackage{amssymb}
\usepackage{amsthm}
\usepackage{amsmath}
\usepackage{amsxtra}
\usepackage{xcolor}

.tfm
.tfm

\theoremstyle{plain}
\newtheorem{prop}{Proposition}[section]
\newtheorem{thm}[prop]{Theorem}

\newtheorem{lemma}[prop]{Lemma}

\newtheorem*{thmA}{Theorem}

\theoremstyle{definition}
\newtheorem{defi}[prop]{Definition}

\theoremstyle{remark}

\newtheorem{remark}{Remark}

\numberwithin{table}{section}

\DeclareMathOperator{\Frob}{Frob}

\DeclareMathOperator{\HT}{HT}
\DeclareMathOperator{\WD}{WD}
\DeclareMathOperator{\Gal}{Gal}

\DeclareMathOperator{\Ind}{Ind}

\newcommand{\R}{\mathscr R}

\newcommand{\F}{\mathbb F}

\newcommand{\Om}{{\mathscr{O}}}

\newcommand{\GL}{{\rm GL}}

\newcommand{\PGL}{{\rm PGL}}

\def\FF{\mathbb F}

\def\CC{\mathbb C}

\def\<#1>{{\left\langle{#1}\right\rangle}}

\def\Z{{\mathbb Z}}             
\def\Q{{\mathbb Q}}             

\def\kro#1#2{\left(\frac{#1}{#2}\right)} 

\def\id#1{{\mathfrak{#1}}}      

\DeclareMathOperator{\norm}{{\mathscr N}}

\DeclareMathOperator{\trace}{{\mathrm{Tr}}}

\begin{document}

\title[Serre's Conjectures]{A simplified proof of Serre's conjecture}

\author{Luis Victor Dieulefait}
\address{Departament de Matem\`atiques i Inform\`atica, Facultat de Matem\`atiques i Inform\`atica,
  Universitat de  Barcelona, Gran Via de les Corts Catalanes,
  585. 08007 Barcelona, Spain}
\email{ldieulefait@ub.edu}
\thanks{}

\author{Ariel Martín Pacetti}
\address{Center for Research and Development in Mathematics and Applications (CIDMA),
Department of Mathematics, University of Aveiro, 3810-193 Aveiro, Portugal}
\email{apacetti@ua.pt}
\thanks{Partially supported by FonCyT BID-PICT 2018-02073 and by
the Portuguese Foundation for Science and Technology (FCT) within
project UIDB/04106/2020 (CIDMA)}
\dedicatory{To the memory of Jean-Pierre Wintenberger}
\date{}

\begin{abstract}
  The purpose of the present article is to present a simplified proof of
  Serre's modularity conjecture using the strong modularity lifting results
  currently available.
\end{abstract}

\maketitle


\section*{Introduction}

In \cite{MR0382173} (Section 3, Problem 1) Serre posted the following
question:

\vspace{2pt}

\noindent {\bf Question:} is it true that any odd, continuous, irreducible two
dimensional Galois representation of the absolute Galois group of
$\mathbb{Q}$ (denoted by $\Gal_\Q$) defined over a finite field is
obtained as the reduction of the $p$-adic Galois representation
attached to a modular form?

\vspace{2pt}

Recall that a $2$-dimensional representation of $\Gal_\Q$
is called \emph{odd} if the image of complex conjugation has
determinant $-1$.  This question is known as ``Serre's weak modularity
conjecture''.  In his famous 1987 paper (\cite{MR885783}) he went
further, giving a precise recipe for a level and weight (minimal in
certain sense) where the modular form should appear. This second
conjecture is known as ``Serre's strong modularity conjecture''. A
very nice reference for the precise statement of the strong modularity
conjecture is the lecture notes \cite{MR1860042}. The
equivalence between Serre's strong and weak version is due to many
authors, the main contributions on the weight reduction being due to
Edixhoven (see \cite{MR1176206} and the references therein). The level
reduction is mainly due to Ribet (see \cite{MR1047143,MR1094193} see
also \cite{MR1265566}). For this reason, we will focus on proving
Serre's weak modular conjecture, namely.

\begin{thmA}[Serre's modularity conjecture] Let
  $\rho:\Gal_\Q \to \GL_2(\overline{\F_p})$ be an odd continuous
  irreducible Galois representation. Then $\rho$ is modular,
  i.e. there exists a modular form
  $f \in S_k(\Gamma_0(N),\varepsilon)$ such that
  $\rho \simeq \overline{\rho_{f,p}}$.
\end{thmA}

\begin{remark}
  In the present article (and following Serre's original conjecture
  given in \cite{MR885783}) we will only consider the so called
  ``cohomological'' modular forms, i.e. those whose weight $k$ is at least $2$ (so that
  they appear in the cohomology of a Shimura curve).  However, the
  notion of ``modularity'' in the last statement could also include
  modular forms of weight one, since the congruence proved in
  \cite{MR379379} $\S 6.9$ shows that if the representation $\rho$ is
  congruent to a representation coming from a weight $1$ modular form,
  then it is also congruent to the representation attached to a
  cohomological modular form, so we can (and will) restrict to weights
  $k \ge 2$.
\label{rem:weight1}  
\end{remark}

The first cases of the conjecture (small level and weight) were proved
by Khare and Wintenberger (\cite{MR2480604}), and also independently
by the first named author (\cite{MR2414504}). A complete proof of the
conjecture was then given by Khare and Wintenberger in
\cite{KW,MR2551764} (a proof for the case of odd level
was also given by the first named author in \cite{MR2959671}). The proof is
based on a smart inductive argument involving both a level and a
weight reduction. The main issue when the proof appeared was that
different modularity lifting results required many technical
conditions (specially while manipulating representations whose
residual image is reducible). Such conditions have been removed during
the last years, which allows us to present a more elegant inductive
argument. In particular, our argument becomes much simpler because the
sophisticated process of weight reduction disappears.  The procedure
(that will be explained in detail in
Section~\ref{section:proofoddcase}) to prove the conjecture when $p$
is odd is the following:
\begin{itemize}
\item Make the representation $\rho$ part of a compatible system
  $\{\rho_{\ell}^{(0)}\}$ (whose definition is
  recalled in Section~\ref{section:lifts}).
  
\item Add a large prime $N$ to the level of the family. This prime is
  needed to ensure that $p$-th member of the family (and the families
  appearing in the ``chain'' of congruences) has large residual image
  for each prime $p$ dividing the ``level'' of the family (except at
  $N$ itself). The prime number $N$ is called a \emph{good-dihedral
    prime} in \cite{KW}.
\item If a prime $p \neq N$ is in the level of the system, remove it
  by looking at the reduction modulo $p$ of the $p$-th member of the
  family and taking a minimal lift (the prime $p=2$ is handled in a
  similar way, with extra technicalities).
  
\item Remove the remaining prime $N$ from the level via the same
  procedure (strong modularity lifting results assure this can be done
  even if the residual image is reducible). Now we are left with a family
  of level $1$ (but without any control on the weight).
  
\item The reduction of the $5$-th member of the last family is either
  reducible (so it is modular), or it is irreducible with Serre
  weight (up to twist) $2$, $4$ or $6$. In the first two cases, taking
  a lift with the right weight, and moving (through a compatible
  family) to the prime $p=3$ (as explained in {\bf Paso 6}) allows us to
  reach the base case of level $1$ proved by Serre.  If Serre's weight
  at the prime $5$ is $6$, our representation is the reduction of a
  representation attached a semistable abelian variety unramified
  outside $5$, corresponding to a base case proved by Schoof.
\end{itemize}
It is important to emphasize that the last step, which substitutes the
``weight reduction", is only possible due to two very strong and
general modularity lifting theorems: a result of Kisin in the
residually irreducible case and a recent result of Pan in the
residually reducible case. The proof for $p=2$ is based on a reduction
to the odd case.

The article is organized as follows: the first section is the most
technical one. It contains the main results (mostly different
modularity lifting theorems) needed to prove Serre's modularity
conjectures. One of the goals of the present article is to allow a
non-expert reader to learn the ideas behind the proof of Serre's
conjectures taking for granted the results of this
section. However, we included in the first section two lemmas
(~\ref{lemma:equivalence} and \ref{lemma:nobaddihedral}) on properties
of residual Galois representations that are well known to experts, but
whose detailed proof is hard to find in the literature. The second
section contains the proof of Serre's conjectures in the case of odd
characteristic, filling in the details of the previous sketch. The
last section contains the proof in characteristic two.

\vspace{5pt}

\noindent{\bf Acknowledgments:} we thank Professor Vytautas
Pa\v{s}k\={u}nas for pointing out an improvement of modularity lifting
theorems at $p=2$ which allowed us to simplify the Paso 4 of our
proof. We also thank the anonymous referee for many suggestions that
improved the quality of the present article.

\section{Main results involved in the proof}

\subsection{Base cases}
The base cases we rely on to \emph{propagate} modularity are the
following.

\begin{thm}
  \label{thm:basecases}
  There are no continuous, odd, absolutely irreducible two dimensional Galois
  representations of the absolute Galois group of $\Q$ unramified outside $p$ and with values on  a finite field of characteristic $p$, for $p=2$ or $3$. 
\end{thm}
\begin{proof}
  The result was proved by Tate (\cite{MR1299740}) for $p =2$, and later Serre observed that the exact same argument worked for $p=3$ (\cite[page 710]{MR3223094}). 
\end{proof}

\begin{thm} There do not exist  non-zero semistable rational
  abelian varieties that have good reduction outside $\ell$ for
  $\ell = 2, 3, 5, 7, 13$.
\label{thm:Schoof}
\end{thm}
\begin{proof}
See \cite[Theorem 1.1]{MR2148199}.
\end{proof}

\begin{thm}[Langlands-Tunnell]
  Let $p$ be an odd prime and $\rho:\Gal_\Q \to \GL_2(\overline{\F_p})$
  be an odd, continuous, irreducible representation with solvable image. Then
  $\rho$ is modular.
  \label{thm:solvable}
\end{thm}

\begin{proof}
  By Dickson's classification, any solvable subgroup of
  $\PGL_2(\overline{\F_p})$ is isomorphic to a cyclic group, a
  dihedral group, $A_4$ or $S_4$. In all cases, the representation
  lifts to an odd representation of $\GL_2(\CC)$. The cyclic case
  gives a reducible representation (and can be omitted). The dihedral
  case was considered by Hecke, the tetrahedral case was proved by
  Langlands in \cite{MR574808} and the octahedral by Tunnell in
  \cite{MR621884}. Note that all the aforementioned results provide a
  weight one modular form (which is enough for our purposes by
  Remark~\ref{rem:weight1}).
\end{proof}

%
\subsection{Modularity lifting Theorems}
Let $\rho:\Gal_\Q \to \GL_2(\overline{\Q_p})$ be a continuous Galois representation.

\vspace{2pt}
\noindent {\bf Problem:} how can we decide whether $\rho$ matches the representation attached to a modular form?

\vspace{2pt}

The term ``Modularity lifting Theorem'' refers to results (like the
ones given in the pioneering articles \cite{MR1333035,MR1333036}) that
provide an answer to the problem. Most modularity lifting theorems
have as a key hypothesis that the residual representation
$\overline{\rho}$ (obtained as the reduction of $\rho$ modulo $p$)
matches the reduction of a representation coming from a modular
form. If the residual representation $\overline{\rho}$ happens to be
reducible (a very hard case of study), it matches the representation
attached to an Eisenstein series, so it is already ``residually modular''.

If $p$ is an odd prime, let $p^\star = \kro{-1}{p} p$, so that the
extension $\Q(\sqrt{p^\star})$ is the unique quadratic extension of
$\Q$ unramified outside $p$.

\subsubsection{Modularity lifting Theorems for residually irreducible representations}

\begin{thm}
  \label{thm:R=T}
  Let $p$ be an odd prime and $\rho:{\Gal_\Q} \to \GL_2(\overline{\Q_p})$
  be a continuous, odd Galois representation ramified at finitely many
  primes and satisfying all the following hypothesis:
  \begin{itemize}
  \item The residual restriction
    $\overline{\rho}|_{\Gal_{\Q(\sqrt{p^\star})}}$ is absolutely
      irreducible,
      
    \item The representation $\rho|_{\Gal_{\Q_p}}$ is de Rham with
      Hodge-Tate weights $\{0,k-1\}$, with $k > 1$,
      
    \item The residual representation is modular,
      i.e. $\overline{\rho} \simeq \overline{\rho_f}$.
    \end{itemize}
    Then $\rho$ matches the representation of a weight $k$ modular form.
\end{thm}

\begin{proof}
  The case $k=2$ is proven in \cite{Kisinann} (Theorem in the second
  page), while the general case follows from \cite{MR2505297} (also
  stated as Theorem in the second page). There are two extra
  hypothesis in the last result: the second one (related to a
  compatibility between classical and $p$-adic local Langlands
  correspondence, as explained in Hypothesis (1.2.6) of
  \cite{MR2505297}) is removed in \cite[Theorem 1.2.1]{MR2251474} and
  in \cite[Theorem 1.1]{MR3306557}. Our precise statement incorporates the results in
  \cite[Theorem 1.4]{Hu-Tan}, where the other hypothesis in Kisin's
  article is removed for $p \ge 5$ together with \cite{1803.07451} (main theorem) where it is removed for $p=3$.
\end{proof}

We also need a similar result for $p=2$.

\begin{thm}
  \label{thm:R=Tat2}
  Let $\rho:{\Gal_\Q} \to \GL_2(\overline{\Q_2})$
  be a continuous, odd Galois representation ramified at finitely many
  primes and satisfying the following hypothesis:
  \begin{itemize}
    \item The representation $\rho|_{\Gal_{\Q_2}}$ is de Rham with
      Hodge-Tate weights $\{0,k-1\}$, with $k > 1$,
      
  \item The residual representation $\overline{\rho}$ is modular and
    has non-solvable image.
    \end{itemize}
    Then $\rho$ matches the representation of a weight $k$ modular form.
\end{thm}

\begin{proof}
  See \cite[Theorem 0.1]{MR2551765} for $k=2$ and $\rho$ potentially Barsotti-Tate, \cite[Theorem 1.1]{MR3544298} and \cite[Theorem A]{MR4257080} for the general case.
\end{proof}

\subsubsection{Modularity of residually reducible representations}

\begin{thm}
  \label{thm:modularityreduciblecase}
  Let $p \ge 5$ be a prime number and
  $\rho:{\Gal_\Q} \to \GL_2(\overline{\Q_p})$ be a continuous,
  irreducible odd Galois representation ramified at finitely many
  primes and satisfying the following two hypothesis:
  \begin{itemize}
    \item The representation $\rho|_{\Gal_{\Q_p}}$ is de Rham with
      Hodge-Tate weights $\{0,k-1\}$ and $k>1$,
      
    \item The semisimplification of $\overline{\rho}$ is a sum of two characters $\overline{\chi}_1 \oplus \overline{\chi}_2$.
    \end{itemize}
    Then $\rho$ matches the representation of a weight $k$ modular form.
\end{thm}
\begin{proof}
  See \cite{SW}(Theorem in the third page) and \cite[Theorem 1.0.2]{1901.07166}.
\end{proof}

\begin{thm}
  \label{thm:SkinnerWiles}
  Let $\rho:{\Gal_\Q} \to \GL_2(\overline{\Q_3})$ be a continuous,
  irreducible odd Galois representation ramified at finitely many
  primes and satisfying all the following hypothesis:
  \begin{itemize}
  \item $\overline{\rho}^{\text{ss}} \simeq 1 \oplus \chi_3$,
    
  \item $\rho|_{D_3} \neq \left(\begin{smallmatrix} 1 & 0 \\ 0 & 1\end{smallmatrix}\right)$,
    
  \item $\rho|_{I_3}\simeq \left(\begin{smallmatrix} * & *\\ 0 & 1\end{smallmatrix}\right)$,
  \item $\det(\rho) = \psi \chi_3^{k-1}$ for some $k \ge 2$.
    \end{itemize}
    Then $\rho$ matches the representation of a weight $k$ modular form.
\end{thm}
\begin{proof}
  See \cite{SW} Theorem in the third page.
\end{proof}

\subsection{Existence of lifts}
\label{section:lifts}
Let $\overline{\rho}:\Gal_\Q \to \GL_2(\F_q)$ be a Galois representation.

\vspace{2pt}
\noindent {\bf Problem:} does there exist a continuous representation $\rho:\Gal_\Q \to \GL_2(\overline{\Q_p})$ whose residual representation is isomorphic to $\overline{\rho}$?
\vspace{2pt}

Note that the representation $\rho$ (if it exists) is far from being
unique. Any such representation is called a \emph{lift} of
$\overline{\rho}$. While working with deformation rings, one studies
lifts into more general coefficient rings, but for our purposes it is
enough to restrict to finite extensions of $\Q_p$.

Let $\overline{\rho}:\Gal_\Q \to \GL_2(\overline{\F_p})$ be an odd,
continuous, irreducible Galois representation, and let
$k(\overline{\rho})$ be the weight of the residual representation as
defined by Serre in \cite{MR885783} (formulas $(2,2,4)$, $(2,3,2)$,
$(2,4,5)$, $(2,4,8)$ and $(2,4,9)$).

\begin{remark}
  Given $\psi: \Gal_\Q \to \overline{\F_p}^\times$ a continuous
  character, it always has a lift to $\overline{\Q_p}^\times$ (for
  example by taking its Teichm\"uller lift). Then the veracity of
  Serre's conjecture for a representation $\overline{\rho}$ is
  equivalent to the veracity of the twist of $\overline{\rho}$ by
  $\psi$. Following Serre's notation, there is always a twist of our
  representation $\overline{\rho}$ such that in formulas $(2,2,4)$,
  $(2,3,2)$, $(2,4,5)$, $(2,4,8)$ and $(2,4,9)$ of \cite{MR885783} we
  can take $a=0$, so we can (and will during the present article)
  abusing notation assume that the weight $k(\overline{\rho})$ is at
  most $p+1$ if $p$ is odd and at most  $4$ if $p=2$.
\label{remark:weightk+1}
\end{remark}

The inductive argument in the proof of Serre's conjecture depends on
reducing (through a combination of taking lifts, building families and creating congruences) the number of ramified primes of the representation
$\overline{\rho}$ (``killing the level''), so we need to be able to 
impose extra conditions on the lift.
Let $E/\Q_p$ be a finite extension and let $\Om_E$ denote its ring of integers.

\begin{defi}
  Let $\ell$ be a prime number, and let $I_\ell$ be the inertia
  subgroup of $\Gal(\overline{\Q_\ell}/\Q_\ell)$. An \emph{inertial
    type} at $\ell$ is a continuous representation
  $\tau_\ell:I_\ell \to \GL_2(E)$.
\end{defi}

For a prime $\ell \neq p$, we will be mostly concerned with the following two inertial types:
\begin{itemize}
\item the \emph{unramified} type is the one corresponding to the trivial representation.
  
\item the \emph{Steinberg} type, corresponding to the restriction to
  the inertial subgroup of the representation of $\Gal_{\Q_\ell}$
  sending a generator of the tame inertia to
  $\left(\begin{smallmatrix} 1 & 1 \\ 0& 1\end{smallmatrix}\right)$,
  and a Frobenius element to
  $\left(\begin{smallmatrix} \ell & 0 \\ 0 &
      1\end{smallmatrix}\right)$.
\end{itemize}

In order to define the inertial type at $p$ (i.e. $\ell = p$) it is better to consider
representations of the Weil-Deligne group of $\Gal_{\Q_p}$ (see
\cite{MR546607}). A representation of the Weil-Deligne group consists
of pairs $(\tau,N)$ where:
\begin{enumerate}
\item $\tau:W(\Q_p) \to \GL_2(\CC)$ is
  a $2$-dimensional complex representation of the Weil group,
  
\item $N$ is a nilpotent endomorphism of $\CC^2$ such that
  \[
w N w^{-1} = \omega_1(w) N, \text{ for all }w \in W(\Q_p),
\]
where $\omega_1$ is the unramified quasi-character giving the action of $W(\Q_p)$ on roots of unity.
\end{enumerate}

The unramified type at $p$ is defined to be the $p$-adic Galois
representation whose Weil-Deligne representation equals
$(\omega_1^{k-1}\oplus 1,\left(\begin{smallmatrix} 0 & 0 \\ 0 &
    0\end{smallmatrix}\right))$, while the Steinberg type at $p$ is
defined to be a twist of
$(\omega_1\oplus 1,\left(\begin{smallmatrix} 0 & 1\\ 0 &
    0\end{smallmatrix}\right))$. The unramified type at $p$ corresponds
to the notion of a \emph{crystalline} representation in p-adic Hodge
theory.

Let $\Sigma$ be a finite set of primes containing $p$ and the primes
where $\overline{\rho}$ is ramified. We want to impose an inertial
type condition on deformations of $\overline{\rho}$ at each prime of
$\Sigma$. For that purpose, for each $\ell \in \Sigma$, $\ell \neq p$,
let $\tau_\ell$ be an inertial type compatible with $\overline{\rho}$,
i.e. such that there exists an $\Om_E$-lattice $\Lambda_\ell$ in $E^2$
which is stable by $\tau_\ell$ (so the choice of a basis for
$\Lambda_\ell$ provides a representation
$\tau_\ell: I_\ell \to \GL_2(\Om_E)$) such that
$\overline{\tau_\ell} = \overline{\rho}|_{I_\ell}$.

For the purposes of the present article, a ``minimal lift'' is a
lift that is unramified at all primes $\ell \neq p$ where the
representation $\overline{\rho}$ is unramified. 

\begin{thm}
  \label{thm:lifting}
  Let $\overline{\rho}:\Gal_\Q \to \GL_2(\overline{\F_p})$ be an
  odd, continuous, representation whose restriction to $\Gal_{\Q(\zeta_p)}$ is
  absolutely irreducible. Assume furthermore that when $p=2$
  $\bar{\rho}$ has non-solvable image. Then there exists a lift
  $\rho:\Gal_\Q \to \GL_2(E)$ (for some finite extension $E/\Q_p$) with
  any of the following prescribed properties:
  \begin{enumerate}
  \item If $p=2$ and $k(\overline{\rho}) = 2$, then $\rho$ is a minimal crystalline lift with Hodge-Tate weights $\{0,1\}$.
    
  \item If $p=2$ and $k(\overline{\rho})=4$, then $\rho$ is a lift
    with Hodge-Tate weights $\{0,1\}$, minimally ramified outside $2$
    and the inertial Weil-Deligne parameter at $2$ is given by
    $(\omega_1 \oplus 1,\left(\begin{smallmatrix}0 & 1\\ 0 &
        0\end{smallmatrix}\right))$.

    \item If $p>2$, then $\rho$ is a minimal crystalline lift with
      Hodge-Tate weights $\{0,k(\overline{\rho})-1\}$.
    
    \item If $p>2$, then $\rho$ is a lift with Hodge-Tate weights
      $\{0,1\}$, any inertial type $\tau_\ell$ compatible with
      $\overline{\rho}$ at primes $\ell \neq p$ in $\Sigma$ and
      unramified outside $\Sigma$. Furthermore, if
      $k(\overline{\rho})=2$ the lift can be taken to be crystalline
      at $p$ and if $k(\overline{\rho})=p+1$ the lift can be taken to
      be Steinberg at $p$, i.e., its inertial Weil-Deligne parameter
      at $p$ is given by
      $(\omega_1 \oplus 1,\left(\begin{smallmatrix}0 & 1\\0 &
          0\end{smallmatrix}\right))$
  \end{enumerate}
\end{thm}

\begin{proof}
  The first three cases are due to Khare-Wintenberger (\cite[Theorem
  5.1]{KW}, its proof given in \cite{MR2551764}). Partial results of
  the last case are also proven in Khare-Wintenberger's article (same Theorem), the more general case is due to Gee
  and Snowden (\cite{MR2785764};  \cite[Theorem
  7.2.1]{Snowden}).
\end{proof}

\subsection{Existence of almost strictly compatible systems} One
important property of Galois representations coming from modular forms
is that they come in ``families''. More concretely, if
$f \in S_k(\Gamma_0(N),\varepsilon)$ is a newform, then (by
\cite{MR3077124}) for every prime number $p$, there exists a
continuous odd Galois representation
$\rho_{f,p}:\Gal_\Q \to \GL_2(\overline{\Q_p})$ unramified outside
$Np$, such that for any prime number $\ell \nmid Np$, the characteristic
polynomial of $\rho_{f,p}(\Frob_\ell)$ equals
$x^2-a_\ell(f) x + \varepsilon(\ell)\ell^{k-1}$, where $a_\ell$
denotes the  eigenvalue of $f$ for the action of the Hecke operator $T_\ell$. Furthermore, one
can obtain information at primes $\ell$ dividing $Np$ as well.

There is a notion of compatible families for abstract representations
(as given by Serre in \cite{serre-book}, I-11). Recall the definition
of a strictly compatible system and an almost strictly compatible
system of Galois representations given in \cite{KW}.

\begin{defi}
  A rank $2$ \emph{strictly compatible system} of Galois
  representations $\R$ of $\Gal_\Q$ defined over $K$ is a $5$-tuple
\[
\R=(K,S,\{Q_{\ell}(x)\},\{\rho_{\id{p}}\},k),
\]
where
\begin{enumerate}
\item $K$ is a number field.
\item $S$ is a finite set of primes.
\item for each prime $\ell \not \in S$, $Q_{\ell}(x)$ is a degree
  $2$ polynomial in $K[x]$.
\item For each prime ideal $\id{p}$ of $K$, the representation
\[
\rho_{\id{p}}:\Gal_\Q \to \GL_2(K_{\id{p}}),
\]
is a continuous semisimple representation such that:
\begin{itemize}
\item If $\ell \not \in S$ and $\ell \nmid \norm(\id{p})$ (the norm of $\id{p}$), then
  $\rho_{\id{p}}$ is unramified at $\ell$ and
  $\rho_{\id{p}}(\Frob_{\ell})$ has characteristic polynomial
  $Q_{\ell}(x)$.
\item If $\ell \mid \norm(\id{p})$, then $\rho|_{\Gal_{\Q_{\ell}}}$ is de Rham and furthermore crystalline if $\ell \notin S$.
\end{itemize}
\item The Hodge-Tate weights $\HT(\rho_{\id{p}})=\{0,k-1\}$.
\item For each prime $\ell$ there exists a Weil-Deligne representation
  $\WD_{\ell}(\R)$ of $W_{\Q_{\ell}}$ over $\overline{K}$ such that for each
  place $\id{p}$ of $K$  and every $K$-linear
  embedding $\iota:\overline{K} \hookrightarrow \overline{K}_{\id{p}}$,
  the push forward
  $\iota\WD_{\ell}(\R) \simeq
  \WD(\rho_{\id{p}}|_{\Gal_{\Q_{\ell}}})^{K\text{-ss}}$.
\end{enumerate}
\end{defi}

\begin{remark}
  Comparing to the case of representations coming from a newform
  $f \in S_k(\Gamma_0(N),\varepsilon)$, the set $S$ consists of the
  primes dividing $N$. Deligne's result implies that $\rho_{f,p}$
  satisfies the third hypothesis and the first item of the fourth one,
  while the last one is the compatibility at the primes dividing $Np$.
\end{remark}

An \emph{almost strictly compatible system} is a $5$-tuple satisfying
the first five properties, and also condition $(6)$ but with
some exceptions: for a prime $\lambda$ whose residual characteristic
is equal to the prime $p$, if the residual representation
$\bar{\rho}_\lambda$ is reducible, then we only impose the
compatibility as in condition $(6)$ if this prime $p$ is odd and the
representation $\WD_{p}(\R)$ is unramified.

\begin{thm}
  Let $\rho:\Gal_\Q \to \GL_2(K_\lambda)$ be an odd, irreducible,
  continuous Galois representation ramified at finitely many places
  and de Rham at $p$ with Hodge-Tate weights $\{0,k-1\}$, with $k > 1$
  such that the restriction of $\bar{\rho}$ to $\Gal_{\Q(\zeta_p)}$ is
  absolutely irreducible. Assume furthermore that when $p=2$
  $\bar{\rho}$ has non-solvable image.  Then $\rho$ is part of a rank
  $2$ almost strictly compatible system of Galois representations.
\label{thm:compatible}
\end{thm}
\begin{proof}
  See \cite[Theorem 1.1]{LuisFamilies}.
\end{proof}

\begin{remark}
  \label{remark:modularityfamily} If $\rho$ is part of a compatible
  system of Galois representations $\{\rho_{\id{p}}\}$, then $\rho$ is
  modular if and only if any given member of the family is. The reason
  is the following: if $\rho \simeq \rho_f$, where $f$ is a newform,
  then by Deligne's theorem, there exists a strong compatible system
  $\{\rho_{f,\id{p}}\}$ containing $\rho_f$. Then for any prime
  $\id{p}$, the representations $\rho_{\id{p}}$ and $\rho_{f,\id{p}}$
  have the same trace and determinant at the Frobenius element
  $\Frob_\ell$, for all prime numbers $\ell$ belonging to a density
  one set of primes (actually all primes but finitely many), so by
  the Brauer-Nesbitt theorem, they are indeed isomorphic.
\end{remark}

\subsection{Some lemmas on the image of Galois representations}

\begin{defi}
A residual representation $\bar{\rho}:\Gal_\Q \to \GL_2(\overline{\F_p})$
which is irreducible but becomes reducible while restricted to
$\Gal_{\Q(\sqrt{p^\star})}$ is called \emph{bad dihedral}. 
\end{defi}

Some of the previous theorems are stated under the hypothesis that the
restriction of the residual representation to $\Gal_{\Q(\zeta_p)}$ is
absolutely irreducible. Let us show that this is in fact equivalent to
the condition (introduced by Wiles) of not being bad dihedral.

\begin{lemma}
  Let $p$ be an odd prime and
  $\bar{\rho}:\Gal_\Q \to \GL_2(\overline{\F_p})$ be an odd continuous
  representation. Then the following are equivalent:
  \begin{enumerate}
  \item $\bar{\rho}|_{\Gal_{\Q(\sqrt{p^\star})}}$ is irreducible,
      
    \item $\bar{\rho}|_{\Gal_{\Q(\zeta_p)}}$ is irreducible.
  \end{enumerate}
\label{lemma:equivalence}
\end{lemma}

\begin{proof}
  We can assume that $p \neq 3$ as otherwise the statement is trivial.
  Clearly the second condition implies the first one. For the
  converse, suppose that the restriction of $\bar{\rho}$ to
  $\Gal_{\Q(\zeta_p)}$ is reducible. In particular, the image of
  $\bar{\rho}$ is a solvable group. Then either our representation is
  contained in a Borel group (hence it is reducible), it lies in the
  normalizer of a split Cartan group (as our coefficient field is
  algebraically closed) or its projective image is one of the
  exceptional groups $A_4$ or $S_4$ (it cannot be $A_5$ because it is
  solvable).

  Let $G$ denote the image of $\bar{\rho}$ and suppose it lies in $N$,
  the normalizer of a Cartan group. Recall that $N$ fits into the
  short exact sequence
  \[
    \xymatrix{
      1 \ar[r] &  T \ar[r]& N \ar[r]^\phi & \Z/2 \ar[r] &  1
      }
    \]
    where $T$ is a torus (corresponding to matrices of the form
    $\left(\begin{smallmatrix} a & 0 \\ 0 &
        b\end{smallmatrix}\right)$), and
    $\left(\begin{smallmatrix} 0 & 1\\ 1 & 0\end{smallmatrix}\right)$
    can be taken as a lift of the generator of $\Z/2$. If we intersect
    each term of the exact sequence with the subgroup $G$, we get
    a similar sequence.  Note that the image of $G$ by $\phi$ is
    non-trivial as otherwise the group $G$ would be abelian, and
    $\bar{\rho}$ would not be irreducible. Also note that since $G$ is
    not abelian, it must contain at least one matrix of the form
    $\left(\begin{smallmatrix} a & 0\\ 0 & b\end{smallmatrix}\right)$
    with $a \neq b$.

    Let $H$ be the image of the restriction of $\bar{\rho}$ to
    $\Gal_{\Q(\zeta_p)}$, a normal subgroup of $G$. Since the
    restriction of $\bar{\rho}$ to $\Gal_{\Q(\zeta_p)}$ is reducible,
    its image  must be an abelian group. The reason is that in some
    chosen basis, its image lies in a Borel subgroup, but the
    normalizer of a Cartan group does not have elements of order $p$
    (as its order is not divisible by $p$). The only abelian subgroups
    of $N$ are the ones contained in $T$, or subgroups of the form
    $\left\langle \left(\begin{smallmatrix}a & 0 \\ 0 &
          a\end{smallmatrix}\right),\left(\begin{smallmatrix} 0 & 1\\
          1 & 0\end{smallmatrix}\right)\right\rangle$ (where $a$ lies
    in $\overline{\F_p}^\times$). The latter are not normal
    subgroups of $G$ (since any such group is not preserve under
    conjugation by any matrix of the form
    $\left(\begin{smallmatrix} a & 0 \\ 0 & b\end{smallmatrix}\right)$
    when $a \neq b$), so $H$ must be a subgroup of $T$.

    Let $\sigma \in \Gal_\Q$ be such that it generates the Galois
    group $\Gal(\Q(\zeta_p)/\Q)$. Then
    $G = \langle H, \bar{\rho}(\sigma)\rangle$. The image of the
    restriction of $\bar{\rho}$ to $\Gal_{\Q(\sqrt{p^\star})}$ equals
    $\langle H ,\bar{\rho}(\sigma^2)\rangle$ which also lies in $T$ so
    the restriction of $\bar{\rho}$ to $\Gal_{\Q(\sqrt{p^\star})}$ 
    is also a reducible representation.

    In the other two cases, since we can assume that $p \neq 3$, the projective image
    being reducible is equivalent to it being decomposable (since
    $A_4$ and $S_4$ do not have elements of order $p$ if $p > 3$), in
    which case the image of the projective
    representation has a normal abelian subgroup with abelian
    quotient. There is no such a subgroup for the group $S_4$, while
    $A_4$ only contains the Klein group $\Z/2\times \Z/2$ (consisting
    of the two $2$-cycles) with quotient group $\Z/3$, but there is no
    faithful reducible and projective representation of the Klein
    group.
\end{proof}

Let us state (and give a detailed proof of) a result that is
well-known to experts and will be needed later. As explained in
Remark~\ref{rem:weight1} we assume that (possible after twisting)
all Serre weights lie in the range $[2,p+1]$.

\begin{lemma}
  Let $p$ be an odd prime and
  $\bar{\rho}:\Gal_\Q \to \GL_2(\overline{\F_p})$ be a continuous odd bad
  dihedral representation. Then either $p = 2k(\bar{\rho})-3$ (niveau
  2 case) or $p = 2k(\bar{\rho})-1$ (niveau 1 case).
\label{lemma:nobaddihedral}
\end{lemma}

\begin{proof} This is Lemma 6.2 (ii) of \cite{KW}. The proof lacks
  some details, hence the difference between the niveau 1 and the
  niveau 2 case is hard to see (which justifies the present proof);
  our argument follows the lines of \cite{MR1610883} (see Proposition
  2.2). The fact that the representation has irreducible image
  together with the projective representation being dihedral (which
  corresponds to the image of the representation lying in the normalizer of a  Cartan group) and $p \ge 3$ imply that there are no elements
  of order $p$ in the image of $\bar{\rho}$. In particular, all
  matrices are semisimple and the image of the $p$-th inertia subgroup
  $I_p$ factors through the tame part, so it is a cyclic group.

  We claim that the projective image of $\bar{\rho}(I_p)$ has order at
  most $2$. Suppose on the contrary that it has order $n$ greater than
  two and that our representation $\bar{\rho}$ is bad dihedral. Let
  $L$ denote the field extension fixed by the projective residual
  image, so $\Gal(L/\Q)$ is a dihedral group $D_{2m}$ of order $2m$
  where $n\mid m$. Since our representation is bad dihedral, the
  restriction of the projectivization of $\bar{\rho}$ to the subgroup
  $\Gal(L/\Q(\sqrt{p^\star}))$ is reducible and decomposable (as all
  elements are semisimple). In particular, the group
  $\Gal(L/\Q(\sqrt{p^\star}))$ is a normal abelian subgroup of order
  $m$, i.e. it is the cyclic subgroup $C_m$ of rotations, of index $2$ in $D_{2m}$. Note that there is a unique such subgroup because we are assuming $n >2$, thus $m >2$.

  Since the projective image of inertia is a cyclic group of order
  greater than $2$, it must lie in the subgroup $C_m$ of
  $D_{2m}$ of rotations, so the field fixed by the rotations subgroup
  is on the one hand $\Q(\sqrt{p^\star})$ and on the other one an
  unramified quadratic extension of $\Q$, which is a contradiction.We have thus established that the projective image of $\bar{\rho}(I_p)$ has order at most $2$.

  Start considering the case of a ``niveau 1''
  character, i.e, the image of the inertia group $I_p$ is of the form
  $\left(\begin{smallmatrix} \chi_p^{k(\bar{\rho})-1} & 0\\
      0 & 1
    \end{smallmatrix}\right)$ with $k(\bar{\rho})\le p$ (recall that we have already shown that this image is  abelian). Observe that in this case the order of this group agrees with its projective order. Then we know that
  $\chi^{k(\bar{\rho})-1}$ has order at most $2$ and
  $k(\bar{\rho}) \le p$, so either $2k(\bar{\rho})-2 = p-1$ or
  $k(\bar{\rho})-1=p-1$. In the second case, we have $k(\bar{\rho})=p$, but then 
   the image of inertia is trivial, hence the
  Galois extension corresponding to $\bar{\rho}$ is disjoint from
  $\Q(\sqrt{p^\star})$. In particular, the representation is not bad
  dihedral.

  In the case of a character of niveau 2, the image of the inertia groups $I_p$ is of
  the form $\psi^{k(\bar{\rho})-1} \oplus \psi'^{k(\bar{\rho})-1}$,
  hence the projective order equals the order of $\psi^{(k(\bar{\rho})-1)(p-1)}$,
  which is the $k(\bar{\rho})-1$-st power of a character of order
  $p+1$. It has order at most $2$ when
  $\frac{p+1}{2} \mid k(\bar{\rho})-1$ and the condition
  $k(\bar{\rho}) \le p$ gives that $p = 2k(\bar{\rho})-3$.
\end{proof}

\begin{remark}
  Over $\Q$, a Galois representation with Serre's level $1$ cannot be
  bad dihedral in the niveau 2 case due to a Lemma of Wintenberger
  (see \cite[Lemma 6.2]{MR2254626} part (i)), hence the bad dihedral
  case can only occur for $p=2k(\bar{\rho})-1$ in level $1$. But in this case
  $k(\bar{\rho})$ is even, hence a bad dihedral representation of
  Serre's level $1$ can only occur when $p \equiv 3 \pmod 4$.
\label{rem:nobaddihedral}
\end{remark}

%

The following lemma will also be needed in the proof.

\begin{lemma}
  \label{lemma:dihedral}
  Let $p$ be a prime that is congruent to $1$ modulo $4$ and let
  $\overline{\rho}:\Gal_\Q \to \GL_2(\F_p)$ be an odd representation
  with non-solvable image. Denote by $\overline{\rho}_{\text{proj}}$
  the projectivization of $\overline{\rho}$ and by $c \in \Gal_\Q$ a
  complex conjugation. Then there exists a set of primes $\{q\}$ of
  positive density that are unramified for $\overline{\rho}$ and such
  that
  \begin{enumerate}
  \item $\overline{\rho}_{\text{proj}}(\Frob_q)$ and $\overline{\rho}_{\text{proj}}(c)$ define the same conjugacy class in $\overline{\rho}_{\text{proj}}(\Gal_\Q)$,
    
  \item $q$ is congruent to $1$ modulo all primes $\le p-1$ and $q$ is
    congruent to $1$ modulo $8$,
    
  \item $q$ is congruent to $-1$ modulo $p$.
  \end{enumerate}
\end{lemma}

\begin{proof}
  See \cite[Lemma 8.2]{KW}.
\end{proof}

\section{The proof of Serre's modularity conjecture for $p \neq 2$.}
\label{section:proofoddcase}
We can assume without loss of generality that ${\rho}$ has
non-solvable image, as otherwise modularity follows from
Theorem~\ref{thm:solvable}.  Let $\rho^{(0)}$ be a minimal crystalline
lift of $\rho$, which exists by Theorem~\ref{thm:lifting} (3) and let
$\{\rho^{(0)}_\ell\}$ be an almost strictly compatible system passing
through $\rho^{(0)}$ (Theorem~\ref{thm:compatible}). Let $k = k(\rho)$
be the weight of the system (recall that taking a suitable twist we
always assume that $k \le p+1$ as explained in
Remark~\ref{remark:weightk+1}), i.e., the Hodge-Tate weights of the
representation $\rho^{(0)}_\ell$ are $\{0 , k-1 \}$. Note that by
Remark~\ref{remark:modularityfamily} it is enough to prove that any
member of the family is modular.

\vspace{2pt}
\noindent {\bf Paso 1:} change to a weight two system. Let $w$ be a
sufficiently large prime, so that it does not belong to the
ramification set of the compatible system and such that $w > 2k$. Then
$\rho^{(0)}_w$ is in what is called the \emph{Fontaine-Laffaille} case
(i.e. the representation is crystalline and the weight is smaller than
the residual characteristic), thus Serre's weight of
$\overline{\rho^{(0)}_w}$ coincides with the weight $k$ of the
compatible system. Lemma~\ref{lemma:nobaddihedral} implies that
$\rho^{(0)}_w$ is not residually bad dihedral. Consider the reduction
$\overline{\rho^{(0)}_w}$ of $\rho^{(0)}_w$. If its image happens to be solvable,
we claim that $\rho^{(0)}_w$ is modular, and then so is
$\rho^{(0)}_\ell$ (by Remark~\ref{remark:modularityfamily}). The
reason is that either the residual image of $\rho^{(0)}_w$ is
reducible (in which case modularity follows from
Theorem~\ref{thm:modularityreduciblecase}) or otherwise, it is
irreducible but solvable, in which case the residual representation is
modular by Theorem~\ref{thm:solvable} and $\rho^{(0)}_w$ is modular by
Theorem \ref{thm:R=T}. In this case the proof ends here (this argument
is crucial and will be used later on).

Suppose on the contrary, that the residual representation has
non-solvable image. Take a weight $2$ lift $\rho^{(1)}_w$ of
$\overline{\rho^{(0)}_w}$, with ramification set equal to the one of $\overline{\rho^{(0)}_w}$
(whose existence is guaranteed by Theorem~\ref{thm:lifting} (4)). This representation in general will not be crystalline at $w$. Now
we can focus on proving modularity of the representation
$\rho^{(1)}_w$, since Theorem \ref{thm:R=T} implies that
$\rho^{(0)}_w$ is modular if and only if $\rho^{(1)}_w$ is. This is
the idea of \emph{propagating} modularity, one can move via a suitable
congruence from one representation to another one until we hit a
modular representation.

Using Theorem~\ref{thm:compatible}, make $\rho^{(1)}_w$ part of an
almost strictly compatible system of Galois representations
$\{\rho^{(1)}_\ell\}$ and let $p_1,\ldots,p_r$ be the primes in the
ramification set of the system (set that we denote by $S_1$). 

\vspace{2pt}
\noindent {\bf Paso 2:} add what is called a ``good-dihedral
prime''. This implies adding an extra prime $N$ to the ramification
set, such that the local type of the resulting system at $N$ is dihedral, i.e. the induction of a
character from a quadratic extension.

Let $K$ be the coefficient field of the compatible system, and let
$q \equiv 1 \pmod 4$ be a prime number greater than the primes in
$S_1$ and also greater than $5$ which splits in $K$. By
Lemma~\ref{lemma:nobaddihedral}, the residual representation
$\overline{\rho^{(1)}_q}$ is not bad dihedral (we are using again the fact that we are in the Fontaine-Laffaille situation, now with a system of weight $k =2$).  If the residual image
of $\rho^{(1)}_q$ happens to be solvable, then once again
$\rho^{(1)}_q$ is modular (because of Theorem~\ref{thm:solvable} and
\ref{thm:R=T} in the irreducible case and
Theorem~\ref{thm:modularityreduciblecase} in the reducible case). Then
we can assume that the image of the residual representation
$\overline{\rho^{(1)}_q}$ is non-solvable.

Since a continuous representation always fixes a lattice, after a
possible conjugation, we can assume that the representation
$\rho^{(1)}_q$ takes values in $\GL_2(\Om_K)$ (the ring of integers of
$K$), so its residual representation has image lying in $\GL_2(\F_q)$
(recall that $q$ splits in $K$). Let $N$ be a prime number which does
not belong to the ramification set $S_1$ as in
Lemma~\ref{lemma:dihedral}. The first hypothesis implies that
$\trace(\overline{\rho^{(1)}_q}(\Frob_N)) = 0$, and the last one
implies that $\chi_q(\Frob_N) + 1$ is also congruent to zero, so the
restriction of the residual representation $\overline{\rho^{(1)}_q}$
to the decomposition group at $N$ (up to a twist) is of the form
\[
\overline{\rho^{(1)}_q}|_{D_N} \simeq \begin{pmatrix} \overline{\chi_q} & 0\\ 0 & 1\end{pmatrix}.
\]
Let $\Q_{N^2}$ denote the quadratic unramified extension of $\Q_N$, so
that its residue field has order $(N-1)(N+1)$. Since $q\mid (N+1)$,
there exists (by class field theory) a character
$\varkappa : \Gal_{\Q_ {N^2}} \to \overline{\Z_q}^\times$ of order $q$
(whose restriction to the inertia group at $q$ is precisely a
character of \emph{niveau} 2). Note that the composition of
$\varkappa$ with the residue map is trivial (as
$\overline{\F_q}^\times$ does not have elements of order $q$), so the
induction $\Ind_{\Gal_{\Q_{N^2}}}^{\Gal_{\Q_N}}\varkappa$ is a
dihedral $2$-dimensional representation of $\Gal_{\Q_N}$, whose
residual representation is unramified, and has zero trace at a
Frobenius element. In particular, we can take
$\Ind_{\Gal_{\Q_{N^2}}}^{\Gal_{\Q_N}}\varkappa$ as our local type
condition at the prime $N$.

Take a crystalline lift $\rho^{(2)}_q$ of weight $2$ with such a
dihedral image of order $2q$ at the prime $N$ (whose existence is
guaranteed by Theorem~\ref{thm:lifting} (4)). Theorem \ref{thm:R=T}
implies that the representation $\rho^{(1)}_q$ is modular if and only
if the representation $\rho^{(2)}_q$ is modular. Make $\rho^{(2)}_q$
part of an almost strictly compatible system $\{\rho^{(2)}_\ell\}$,
whose ramification set equals $S=S_1 \cup \{N\}$.

\begin{lemma}
  Let $\{\rho_\ell\}$ be an almost strictly compatible system of
  Galois representations, whose ramification set is contained in
  $S\cup \{2,3\}$, and such that the local inertial type at the
  prime $N$ matches the type of
  $\Ind_{\Gal_{\Q_{N^2}}}^{\Gal_{\Q_N}}\varkappa$ defined before. Then
  for any prime $p \in S_1 \cup\{2,3\}$ the residual representation
  $\overline{\rho_p}$ has non-solvable image.
  \label{lemma:largeimage}
\end{lemma}
\begin{proof}
  This result is taken from \cite[Lemma 6.3]{KW}. Let
  $p \in S_1\cup\{2,3\}$ be a prime, and let $\rho_{p}$ be the
  $p$-th representation of the family. The strong compatibility
  condition implies that its restriction to the decomposition group
  $D_N$ is given as the induction of a ramified character of order $q$
  from the quadratic extension $\Q_{N^2}$. Since the prime $q$ is
  larger than $p$ (by construction), the kernel of the reduction map
  $\overline{\Z_{p}} \to \overline{\FF_{p}}$ does not have elements of
  order $q$, hence the restriction of the residual representation
  $\overline{\rho_p}$ to the decomposition group $D_N$ is irreducible
  (in particular, the residual representation is irreducible).

  If the representation has solvable image, by Dickson's
  classification (and the fact that its projective image has order
  divisible by $q > 5$), its projective image must be a dihedral
  group. Suppose this is the case, so the projectivization of our
  residual representation is induced from a quadratic extension $K$ of
  $\Q$. Note that $K$ can only ramify at primes where our family does,
  namely some primes of $S \cup \{2,3\}$ (recall that
  $p \in S \cup\{2,3\}$ so the ramification at $p$ is also
  considered). Since $N \equiv 1 \pmod{p}$ for all odd primes
  $p \in S_1 \cup\{3\}$ (by Lemma~\ref{lemma:dihedral} and the fact
  that $p < q$), and $N \equiv 1 \pmod 8$, either $N$ splits in $K$ or
  it is ramified. If $N$ splits in $K/\Q$, then the restriction of
  $\overline{\rho_p}$ to the decomposition group $D_N$ would be
  reducible (since $\overline{\rho_p}$ restricted to $\Gal_K$ is
  reducible), contradicting what we proved in the first
  paragraph. Otherwise, if $N$ ramifies in $K/\Q$, the inertia group
  at the prime $N$ has even order contradicting the fact that our
  local type at $N$ is the induction from a quadratic unramified
  extension of an odd order character.
\end{proof}

\begin{defi}
  A continuous representation
  $\rho:\Gal_\Q \to \GL_2(\overline{\Q_p})$ has \emph{large image} if
  its residual representation has non-solvable image.
\end{defi}

All the congruences in Paso 3 and Paso 4 are modulo primes in
$S_1 \cup \{2,3\}$, and the ramification sets of the involved representations are contained in $S\cup \{2,3\}$ so the hypothesis of Lemma~\ref{lemma:largeimage}
will always be satisfied. In particular, the good-dihedral prime $N$
 provides the large image hypothesis needed to apply the results of the first section.

\vspace{2pt}
\noindent {\bf Paso 3:} killing the odd part of the level. Let
$\{p_1, \ldots, p_r,N\}$ denote the set of primes where the system is
ramified, $p_1$ being the smallest one (probably equal to $2$).  For each
$i = 2,\ldots ,r$, apply the following procedure:
\begin{itemize}
\item consider the reduction $\overline{\rho^{(i)}_{p_i}}$ of the $p_i$-th entry
  of the compatible system $\{\rho^{(i)}_\ell\}$,
  
\item Take $\rho^{(i+1)}_{p_i}$ to be a minimal lift of
  $\overline{\rho^{(i)}_{p_i}}$ unramified at $p_i$ (which exists by
  Theorem~\ref{thm:lifting} (3), noting that the hypothesis on the
  restriction of the residual representation to
  $\Gal_{\Q(\zeta_{p_i})}$ being absolutely irreducible is fulfilled
  by Lemma~\ref{lemma:largeimage}),
  
\item Make $\rho^{(i+1)}_{p_i}$ part of an almost strictly compatible
  system $\{\rho^{(i+1)}_\ell\}$ using
  Theorem~\ref{thm:compatible}. Now the ramification set of the new
  family does not contain the prime $p_i$ (but is still contained in
  $S_1 \cup \{2,3\}$ so Lemma~\ref{lemma:largeimage} applies to the
  new family).
\end{itemize}
By Theorem~\ref{thm:R=T} the representation $\rho^{(i)}_{p_i}$ is
modular if and only if the representation $\rho^{(i+1)}_{p_i}$ is
modular. In particular, modularity of the compatible system $\{\rho^{(i)}_\ell\}$
is equivalent to modularity of the compatible system $\{\rho^{(i+1)}_\ell\}$.

If $p_1$ is odd, we apply the same procedure at $p_1$, and end with a
system unramified outside $\{N\}$. In such a case, go to Paso 5.

\vspace{2pt}
\noindent {\bf Paso 4:} removing $2$ from the level. This step follows
the method of \cite{KW}.  The following lemma will prove
crucial. Recall that a local type at $p$ is \emph{Steinberg} if its inertial
Weil-Deligne parameter is given by
$(\omega_1  \oplus 1,\left(\begin{smallmatrix}0 & 1\\ 0 &
    0\end{smallmatrix}\right))$.

\begin{lemma}
  Let $\{\rho_\ell\}$ be an almost strictly compatible system with
  Hodge-Tate weights $\{0,1\}$ which is unramified at $3$. If the
  Weil-Deligne representation at the prime $2$ is (a twist of)
  Steinberg and $\rho_3$ has non-solvable residual
  image then there exists another almost strictly compatible
  system $\{\rho'_\ell\}$ with Hodge-Tate weights $\{0,1\}$,
  having the same ramification set as  $\{\rho_\ell\}$ such that the $3$-adic members of the two systems are congruent,
  and whose Weil-Deligne type at the prime $2$ is given by an order
  $3$ character (in particular has trivial monodromy). Furthermore,
  the modularity of one of the systems is equivalent to modularity of
  the other one.
\label{lemma:technical}
\end{lemma}

\begin{proof}
%
%
  Under our hypothesis, $\rho_3|_{D_2} \simeq \left(\begin{smallmatrix} \chi_3 & * \\
      0 & 1\end{smallmatrix}\right)$ (up to twist). We follow the same
  strategy as we did to incorporate the good-dihedral prime to our
  family. We take a \emph{niveau} 2 character $\chi'$ of order $3$
  (corresponding to the quadratic unramified extension $\Q_4$ of
  $\Q_2$) and consider its induction to $\Gal_{\Q_2}$. This
  corresponds to a local Galois representation $\tilde{\rho}_2$ of
  $\Gal_{\Q_2}$ whose restriction to the inertia group $I_2$ is of the
  form
  $\left(\begin{smallmatrix} \chi' & * \\ 0 &
      \chi'^2\end{smallmatrix}\right)$ (so as a Weil-Deligne
  representation, monodromy is trivial). Note that both
  representations ($\rho_3|_{D_2}$ and $\tilde{\rho}_2$) have the same
  reduction while restricted to the inertia subgroup $I_2$ and the
  traces of both reductions at a Frobenius element are zero.

  Let $\rho'_3$ be a
  crystalline weight $2$ minimal lift of $\overline{\rho_3}$ whose
  local type at $2$ matches that of $\tilde{\rho}_2$ (such a lift
  which exists by Theorem~\ref{thm:lifting} (4)). Let $\{\rho'_\ell\}$
  be an almost strictly compatible system containing $\rho'_3$ (whose
  existence is warranted by Theorem~\ref{thm:compatible}). Modularity
  of the system $\{\rho'_\ell\}$ is equivalent to that of
  $\{\rho_\ell\}$ by Theorem~\ref{thm:R=T}.
\end{proof}

\begin{remark} The previous lemma is needed because existence of a
  crystalline lift of a Galois representation into
  $\GL_2(\overline{\F_2})$ is only proved when Serre's weight is $2$
  (see the hypothesis of Theorem~\ref{thm:lifting} (1) and (2)).  As
  explained in the proof of Theorem 9.1 of \cite{KW}, if we start with
  an almost strictly compatible system of Galois representations, such
  that the residual representation at the prime $2$ has non-solvable
  image and has Serre's weight $k(\overline{\rho_2})=4$, then the new
  family provided by Lemma~\ref{lemma:technical} satisfies that
  $k(\overline{\rho'_2}) = 2$. The reason is that the weight $4$ case
  corresponds to a tr\`es ramifi\'ee situation; in particular the
  residual representation is flat over an extension with even
  ramification index. On the other hand, the ramification type at $2$
  of the representation produced by the lemma was chosen so that it
  becomes flat over a cubic extension, leading to a contradiction.
  \label{remark:weight2}
\end{remark}

Recall from Lemma~\ref{lemma:largeimage} that our good-dihedral
prime at $N$ ensures non-solvable residual image at the prime $2$. The
procedure to remove the prime $2$ from the compatible system's level
consists of applying the following steps:

\begin{enumerate}
\item If $k\left(\overline{\rho^{(r+1)}_2}\right) = 4$, take a minimal weight
  $2$ lift $\rho^{(r+2)}_2$ with Steinberg type at $2$ (which exists
  by Theorem~\ref{thm:lifting} (2)) and make it part of an almost
  strictly compatible system $\{\rho^{(r+2)}_\ell\}$. Modularity of
  one compatible system is equivalent to modularity of the other by
  Theorem~\ref{thm:R=Tat2}. Otherwise,
  $k\left(\overline{\rho^{(r+1)}_2}\right) = 2$ in which case go to
  step ($3$) (to unify notation, denote $r+3$ the index).
  
\item Now the compatible system $\{\rho^{(r+2)}_\ell\}$ has Steinberg
  Weil-Deligne type at two. Change its local type at $2$ via a
  congruence at the prime $3$ as explained in
  Lemma~\ref{lemma:technical}. Such a lemma gives an almost strictly compatible
  system $\{\rho^{(r+3)}_\ell\}$ whose local type at $2$ comes from an
  order $3$ character. The same lemma proves that modularity of the
  compatible system $\{\rho^{(r+2)}_\ell\}$ is equivalent to
  modularity of the compatible system $\{\rho^{(r+3)}_\ell\}$.

\item At this step, in both cases
  $k\left(\overline{\rho^{(r+3)}_2}\right) = 2$ by
  Remark~\ref{remark:weight2}. Let $\{\rho^{(r+4)}_\ell\}$ be an
  almost strictly compatible system containing a minimal crystalline
  weight $2$ lift of $\overline{\rho^{(r+3)}_2}$ 
  (whose existence is warranted by Theorem~\ref{thm:lifting}
  (1)). Modularity of one family is equivalent to modularity of the
  other by Theorem~\ref{thm:R=Tat2}.
\end{enumerate}

Then we end with a system $\{\rho'_\ell\}$ unramified outside
$\{N\}$.

\vspace{2pt}
\noindent{\bf Paso 5:} killing the good-dihedral prime.  The reduction
modulo $N$ of the system $\{\rho'_\ell\}$ lies in one of the following cases:
\begin{enumerate}
\item The representation $\overline{\rho'_N}$ is reducible, in which
  case $\rho'_N$ is modular by
  Theorem~\ref{thm:modularityreduciblecase}. End of the
  proof. Incidentally, as the reader may easily check, this is the
  only step in the whole proof where the fact that the compatible
  systems that we are considering are only known to be {\it almost}
  strictly compatible is relevant. Since the local parameter at $N$ of
  this compatible system is ramified and we are assuming that
  $\overline{\rho'_N}$ is reducible, the conditions in the definition
  of ``almost compatible systems'' imply that we do not know the local
  behavior at $N$ of $\rho'_N$, we only know that it is a de Rham
  representation. Luckily, this is enough for
  Theorem~\ref{thm:modularityreduciblecase} to hold.
\item The representation is irreducible and not bad-dihedral. Take an
  almost strictly compatible system $\{\rho''_\ell\}$ containing a
  minimal crystalline lift of it (such a lift exists by
  Theorem~\ref{thm:lifting} (3)). Observe that the ramification set
  $S$ of this system is the empty set. Modularity of the system
  $\{\rho''_\ell\}$ is equivalent to that of $\{\rho'_\ell\}$ by
  Theorem~\ref{thm:R=T}.
  
\item The representation is bad-dihedral. Serre's level of the
  representation $\overline{\rho'_N}$ equals $1$ (as the system was
  unramified outside $N$), then this case cannot occur by
  Lemma~\ref{lemma:nobaddihedral} and Remark~\ref{rem:nobaddihedral},
  because $N \equiv 1 \pmod 4$.
\end{enumerate}


\vspace{2pt}

\noindent{\bf Paso 6:} reduction of the weight. Move to the prime
$p=5$, where the same three possibilities can occur:
\begin{enumerate}
\item The representation $\overline{\rho''_5}$ is reducible, in which
  case $\rho''_5$ is modular by
  Theorem~\ref{thm:modularityreduciblecase}. End of the proof.
  
\item The representation $\overline{\rho''_5}$ has absolutely
  irreducible image, but it is
  bad-dihedral. Since our family has 
  level $1$, we know by Remark~\ref{rem:nobaddihedral} that this is not possible, since $5 \not \equiv 3 \pmod 4$.

\item The representation $\overline{\rho''_5}$ is irreducible and not
  bad-dihedral. If Serre's weight $k(\overline{\rho''_5})=2$ or $4$,
  take an almost strictly compatible system $\{\tilde{\rho}_\ell\}$,
  whose ramification set is empty, containing a minimal crystalline
  lift of $\overline{\rho''_5}$, which exists by
  Theorem~\ref{thm:lifting} (3) and has Hodge-Tate weights
  $(0,k(\overline{\rho''_5})-1)$. If Serre's weight is $6$, take a
  weight $2$ lift, with Steinberg Weil-Deligne type at $p=5$,
  unramified outside $5$, which exists by Theorem~\ref{thm:lifting}
  (4) and make it part of an almost strictly compatible system
  $\{\tilde{\rho}_\ell\}$. In both cases modularity of the system
  $\{\rho''_\ell\}$ follows from that of the system
  $\{\tilde{\rho}_\ell\}$ by Theorem~\ref{thm:R=T}.

\end{enumerate}

In the cases of weight $2$ and $4$, we look at the $3$-adic
representation $\tilde{\rho}_3$ and we reduce modulo
$3$. Theorem~\ref{thm:basecases} implies that there are no irreducible
residual Galois representations unramified outside $3$, so the
representation $\overline{\tilde{\rho}_3}$ is reducible. Furthermore,
$\tilde{\rho}_3$ is ordinary by \cite[Corollary 4,3]{MR2060368},
because the weight is either $2$ or $4 = 3+1$, so the third hypothesis
of Theorem~\ref{thm:SkinnerWiles} is satisfied. Since Serre's weight
is not $3$, the image of $\overline{\tilde{\rho}}_3|_{D_3}$ is
non-trivial, hence Theorem~\ref{thm:SkinnerWiles} implies that
$\tilde{\rho_3}$ is modular.


The last case to consider is when $\tilde{\rho}_5$ has weight $2$, is
Steinberg at $5$ and unramified at any other prime. In this case, by
\cite[Proposition 9.4.1]{Snowden} the representation $\tilde{\rho_5}$
matches the $5$-adic representation of an abelian variety $A/\Q$ of
$\GL_2$-type. Since $\tilde{\rho}_5$ is Steinberg at $5$ and
unramified at any other prime, the variety $A$ is semistable and has
good reduction outside $5$ contradicting Theorem~\ref{thm:Schoof}.

\vspace{2pt}

We have seen that by a combination of suitable modularity lifting
theorems the modularity of the given representation was deduced by
propagation from the base cases proved by Serre and Schoof and cases
where the residual image is solvable. This concludes the proof of
Serre's conjecture in the case of odd characteristic.

\section{The case $p=2$}
The projective image is either solvable or non-solvable.
Recall that all solvable subgroups of $\PGL_2(\overline{\F_2})$ in the
irreducible case have dihedral projective image (the $S_4$ case does
not occur, while the $A_4$ case corresponds to a Borel subgroup over
$\F_4$ see \cite[Lemma 6.1]{KW}). The dihedral case was proven in \cite{MR1610808}. If the
image is non-solvable, take an odd weight two lift (guaranteed by
Theorem \ref{thm:lifting}(1) and (2)), make it part of an almost strictly compatible system
and reduce modulo a prime greater than $3$ and which is not in the ramification set of the
compatible system  (since the residual Serre weight is $2$ under these conditions, applying Lemma  ~\ref{lemma:nobaddihedral} we see that $p > 2 \cdot 2 -1 = 3 $ is enough to avoid the bad-dihedral situation)  to end in a situation covered before, namely: the
representation is either reducible (so modularity follows from
Theorem~\ref{thm:modularityreduciblecase}) or irreducible in odd
characteristic hence it is covered by the cases of Serre's conjecture we have already solved, in which case  modularity of the system follows from Theorem~\ref{thm:R=T} since we know that the residual representation is not bad-dihedral.

\bibliographystyle{alpha}
\bibliography{biblio}
\end{document}